\documentclass[12pt,reno]{amsart}  
\usepackage{amssymb}
\usepackage{amsmath}
\usepackage{enumitem}
\usepackage{mathrsfs}
\usepackage[english]{babel}
\usepackage[all]{xy}
\usepackage{hyperref}
\usepackage{marginnote}
\usepackage{comment}

\usepackage[backend=bibtex,style=numeric,natbib=true,firstinits=true,doi=false,isbn=false,url=false,eprint=false,maxbibnames=99]{biblatex} 
\usepackage[autostyle=true]{csquotes} 
\usepackage{stmaryrd}
\usepackage{fullpage}
\RequirePackage{mathrsfs}
\theoremstyle{plain}	
\newtheorem{theorem}{Theorem}[section]
\newtheorem{conjecture}[theorem]{Conjecture}
\newtheorem{lemma}[theorem]{Lemma}
\newtheorem{proposition}[theorem]{Proposition}
\newtheorem{corollary}[theorem]{Corollary}

\theoremstyle{definition}
\newtheorem{defn}[theorem]{Definition}

\theoremstyle{remark}
\newtheorem{rem}[theorem]{Remark}

\newcommand{\Lc}{\text{V}}
\newcommand{\Oc}{\mathcal{O}}

\newcommand{\Aa}{\mathbb{A}}
\newcommand{\odeg}{\overline{\deg}}
\newcommand{\CaDiv}{\text{CaDiv}}
\newcommand{\oh}{\overline{h}}

\newcommand{\Xm}{\mathcal{X}}
\newcommand{\Wm}{\mathcal{W}}

\newcommand{\xn}{x_1,\dots,x_n}
\newcommand{\zq}{z_1,\dots,z_q}
\newcommand{\yt}{y_1,\dots,y_t}
\newcommand{\Pda}{P_{0,\alpha},\dots,P_{d-1,\alpha}}
\newcommand{\Qa}{Q_\alpha}
\newcommand{\Pao}{P_{\alpha, 0}}

\newcommand{\Pad}{P_{\alpha, d-1}}

\newcommand{\idp}{\mathfrak{p}}

\newcommand{\OK}{\mathit{O}_K}

\newcommand{\OKs}{\mathit{O}_{K,S}}

\newcommand{\OFt}{\mathit{O}_{F,T}}

\newcommand{\oQ}{\overline{\mathbb{Q}}}

\newcommand{\ki}{\textit{k }}

\newcommand{\sbe}{\subseteq}
\newcommand{\spe}{\supseteq}
\newcommand{\ord}{\text{ord}}
\newcommand{\ita}[1]{\textit{#1}}

\newcommand{\disc}{\text{disc}}

\newcommand{\dtint}{$(D,{O}_{F,T})$\text{-integral }}

\newcommand{\inv}{^{-1}}

\newcommand{\Qbar}{\overline{\QQ}}

\newcommand\PP{\mathbb{P}}

\newcommand\ZZ{\mathbb{Z}}
\newcommand\NN{\mathbb{N}}
\newcommand\QQ{\mathbb{Q}}
\newcommand\RR{\mathbb{R}}
\newcommand\CC{\mathbb{C}}

\usepackage{color}

\definecolor{orange}{rgb}{1,0.5,0}

\title[Finiteness theorems for complements of large divisors]{Finiteness theorems for complements of large divisors}

\author{Philipp Licht}
\address{Philipp Licht \\
Institut f\"{u}r Mathematik\\
Johannes Gutenberg-Universit\"{a}t Mainz\\
Staudingerweg 9, 55099 Mainz\\
Germany.}
\email{plicht05@uni-mainz.de}



\begin{document}

\begin{abstract}   
We prove finiteness results on integral points on complements of large divisors in projective varieties over finitely generated fields of characteristic zero. To do so, we prove a function field analogue of arithmetic finiteness results of Corvaja-Zannier and Levin using Wang's function field Subspace Theorem. We then use a method of Evertse-Gy\H ory for concluding finiteness of integral points over finitely generated fields from known finiteness results over number fields.
\end{abstract}

\maketitle

\thispagestyle{empty}

  \section{Introduction}

In \cite{Lang2, LangIHES}, Lang suggests that Diophantine statements involving rational points over number fields should continue to hold over arbitrary finitely generated fields over $\QQ$; see his question on \cite[p.~202]{Lang2} for a precise question of his. For instance, the work of   Siegel-Mahler-Lang  (see \cite{ Mahler1, Parry,Siegel1}) in the classical setting of rings of ($S$-)integers was extended by  Lang to show that the unit equation
\begin{align*}
u+v=1, \quad u,v\in A^*
\end{align*}
has only finitely many solutions when $A$ is any $\ZZ$-finitely generated integral domain of characteristic zero.  

Our main theorem extends finiteness results of Autissier, Corvaja-Zannier and Levin from number fields to finitely generated fields in a similar vein as Lang did for the unit equation.

Before stating our precise results, we introduce some terminology.       By a \emph{variety over a field $K$}   we mean a finite type integral separated scheme over $K$. Let $k$ be an algebraically closed field of characteristic zero. Following \cite[Definition~7.1]{JBook}, a  variety 
  $X$  over $k$ endowed with a closed subset $\Delta$ is said to be \emph{arithmetically hyperbolic  modulo $\Delta$ over $k$} 
if  there is a  $\ZZ$-finitely generated subring $A\subset k$ and a finite type separated $A$-scheme $\mathcal{X}$ with $\mathcal{X}_k \cong X$ over $k$ (i.e., a  model for $X$ over $A$ \cite[Definition~3.1]{JBook}) such that, for all $\ZZ$-finitely generated subrings $ A'\subset k$ containing $A$,  the set $\mathcal{X}(A')\setminus \Delta$  of $A'$-points  on $\mathcal{X}$ is finite.  Thus, roughly speaking,  a variety $X$ is arithmetically hyperbolic modulo $\Delta$ if it has only finitely many integral points outside $\Delta$. We say that $X$ is \emph{pseudo-arithmetically hyperbolic over $k$} if there exists a proper closed subset $\Delta\subsetneq X$ such that $X$ is arithmetically hyperbolic modulo $\Delta$, and we say that $X$ is \emph{arithmetically hyperbolic over $k$} if it is arithmetically hyperbolic modulo the empty subset.  The notion of arithmetic hyperbolicity is independent of the chosen model in the sense that, if $X$ is arithmetically hyperbolic over $k$, then, for every $\mathbb{Z}$-finitely generated subring $A\subset k$ and every model $\mathcal{X}$  for $X$ over $A$,  the set $\mathcal{X}(A)$ is finite; see \cite[Lemma~4.8]{JLalg}. These notions extend   Lang's notions \cite{Lang1} of the Mordell/Siegel property for projective/affine varieties; see        
\cite{Autissier1,  Autissier2, vBJK, CLZ, FaltingsComplements, Faltings2, FaltingsLang,JAut, Jav15, JLitt, 	JL,  	JLFano, 	JXie, 	Levin, 	Moriwaki, 	 UllmoShimura, VojtaSub,	Vojta2, VojtaLangExc} for examples of arithmetically hyperbolic varieties.

Following \cite{JLevin},   if $X$ is a variety over an algebraically closed field $k$ of characteristic zero such that $X_L$ is arithmetically hyperbolic over $L$ for all algebraically closed field extensions $L\supset k$, then we say that $X$ is \emph{absolutely arithmetically hyperbolic}. More generally, if $\Delta \sbe X$ is a closed subset, then $X$ is \emph{absolutely arithmetically hyperbolic modulo $\Delta$} if, for every algebraically closed field $L \spe k$, the variety $X_L$ is arithmetically hyperbolic modulo $\Delta_L$. 
 
Motivated by Lang's aforementioned philosophy on rational points over finitely generated fields, we are interested in the \emph{persistence} of arithmetic hyperbolicity along field extensions. For example, Siegel-Mahler-Lang showed that $\mathbb{A}^1\setminus \{0,1\}$ is arithmetically hyperbolic over $\overline{\mathbb{Q}}$, and Lang showed that this persists over all algebraically closed fields of characteristic zero, thereby proving the so-called ``Persistence Conjecture''  in this case:

\begin{conjecture}[Persistence Conjecture]
Let $k$ be an algebraically closed field of characteristic zero, and let $X$ be a variety over $k$.
If $X$ is arithmetically hyperbolic over $k$, then $X$ is absolutely arithmetically hyperbolic.
\end{conjecture}

For projective varieties, the Persistence Conjecture is a consequence of   Lang's conjectures as formulated in \cite[Section~12]{JBook}. Indeed, if $X$ is a projective arithmetically hyperbolic variety over $k$, then every subvariety of $X$ is of general type by this conjecture, which one can show (using, for example, the results in \cite{JV}) implies that every subvariety of $X_L$ is of general type, so that (again by Lang-Vojta's conjecture), the variety $X_L$ is arithmetically hyperbolic over $L$.  

This conjecture was first systematically studied in \cite[Conjecture~1.5]{JAut} (see also \cite[Conjecture~1.20]{vBJK} and \cite[Conjecture~17.5]{JBook}).
It was shown to hold for algebraically hyperbolic projective varieties in \cite[Section~4]{JAut}, and then also for varieties which admit a quasi-finite morphism to a semi-abelian variety \cite[Theorem~7.4]{vBJK}. It was also shown to hold for varieties which admit a quasi-finite period map \cite{JLitt}, hyperbolically embeddable smooth affine varieties \cite{JLevin} and certain moduli spaces of polarized varieties \cite{JSZ}.

In this paper, we extract new results on the Persistence Conjecture from the work of   Evertse-Gy\H ory  (discussed further below). 
 
  \subsection*{Main Results}

To state our first finiteness result, we consider a result of Levin for varieties over number fields \cite[Theorem~6.1A(b), Theorem~6.2A(d)]{Levin}, \cite[Theorem~1.4]{HL} (after work of Corvaja-Zannier \cite{CZ2,CZ3,CZ4}; see also the work of Autissier \cite{Autissier1, Autissier2}) and extend the result to finitely generated fields.  
  
  \begin{theorem}[Main Result I]\label{thm:I}
Let $m$ be a positive integer, let $X$ be a smooth projective variety over $\Qbar$, and let $D=\sum_{i=1}^rD_i$ be a sum of $r$ ample effective divisors on $X$ such that at most of $m$ of the $ D_i $ meet in any point, where $ r > 2m \dim(X) $. Then the affine variety  $X\setminus D$ is absolutely arithmetically hyperbolic.
  \end{theorem}
  
	In \cite{JLevin}, Javanpeykar-Levin prove Theorem \ref{thm:I} by 1) verifying that   $X\setminus D$ is hyperbolically embeddable (over $\mathbb{C}$) and 2) verifying that the Persistence Conjecture holds for hyperbolically embeddable varieties. Combining 1) and 2) with Levin's theorem then completes the proof. 
	
  Our proof of Theorem \ref{thm:I} proceeds in a different fashion. 
  The main step of Levin's argument is to prove the pseudo-arithmetic hyperbolicity of complements of large divisors \cite[Theorem~8.3A]{Levin}). In the first step, we prove a function field version of that result.

  \begin{defn}[Large divisor]\label{def:large} 
  	Let $X$ be a smooth projective variety over a field $ k $. 
  	An effective divisor $D$ on $X$ is \ita{very large} if for all $ P \in D(\overline{k}) $, there is a basis $ B $ for 
  	$ \Lc(D) = H^0(X,\mathcal{O}_D) $ such that $ \ord_E \left( \prod_{f \in B} f \right) > 0 $ for all irreducible components $E$ of $D$ with $ P \in E(\overline{k}) $. An effective divisor $D$ is \ita{large} if some positive integral linear combination of its irreducible components is very large.
  \end{defn}

  \begin{theorem}[Main Result II, Corvaja-Zannier-Levin for countable function fields] \label{t:corvaja zannier function fields}
  	
  	Let $X$ be a projective variety over $\oQ$ and let $D$ be a very large divisor on $X$.
  	Then there is a proper closed subscheme $ Z \sbe X $ with the following property.
  	
  	For every countable algebraically closed field \ki of characteristic 0, for every algebraic function field $ F $over $k$ (as defined in Subsection \ref{ss:height over functions fields}), for every finite set of places $ T \sbe F$ and for every set of \dtint points $R\sbe (X \setminus D)(F)$ (as defined in Subsection \ref{ss:binding or bounding integral points}),
  	there is subset $ R^\prime \subseteq R $ with the same Zariski closure $ \overline{R} = \overline{R^\prime} $ such that $ h \circ \phi_D $ is bounded on $ R^\prime \setminus Z $.
  \end{theorem}

For the proof, we appeal to the function field version of the Subspace Theorem \cite{AnWang, Wan04} and ``re-do'' some of the arguments in \cite{Levin}. 

In the second step, we apply a method by Evertse and Gy\H ory \cite[Chapter~8]{Eve15}. They studied the unit equation over finitely generated domains. Given 
\begin{enumerate}
	\item an effective bound for the number of solutions over $\OKs$ for any  $K$ and any finite set of places $S$, and
	\item an effective height bound for the solutions over function fields,  
\end{enumerate}
they showed that there is an effective bound for the number of solutions over finitely generated domains. 

To explain their method, we consider the following example. Let $ A = \ZZ[t,1/(1+t^2)] $, and let $R\sbe A$ be a subset. 
Then $ A $ is contained in the function field $ F = \oQ(t) $. For $ p,q \in \ZZ[t] $, $q \neq 0$, with $\mathrm{gcd}(p,q) = 1 $ and $ p/q \in R $ the height is given as 
$$ H_F^{aff}(p/q) = \max\{\deg(p),\deg(q)\} .$$
So bounding the height bounds the degree of possible denominators and numerators. 

Furthermore, evaluating at $u\in \ZZ $  defines a specialization map $ \psi_u\colon A \to \ZZ[1/(1+u^2)] $. 
If $ \psi_u(R) $ is finite for sufficiently many $ u \in \ZZ $, then the set $R$ is finite since each element of $R$ has to interpolate these points.
By extracting their key arguments, we can show that their method works in a more general setting. This culminates in the following result.

\begin{theorem}[Main Result III, Evertse-Gy\H ory method on affine varieties]\label{thm:EvertseGyoryAffine}
	Let $X\sbe \Aa^n_{\oQ}$ be an affine variety over $ \oQ $,
	let $K$ be a number field, let $S$ be a finite set of places of $K$, let $ \mathcal{X} \sbe \Aa^n_{ \OKs}$ be a model for $X$ over $ \OKs$. Let $A\spe \OKs$ be a $\ZZ$-finitely generated integral domain with quotient field $L$, and
	let $ R \sbe \Xm(A) $ be a subset. Assume that 
	\begin{enumerate}
		\item $X$ is arithmetically hyperbolic over $ \oQ $, and
		\item for each  finite set of function fields $ \{ F_1, \dots, F_t \} $ with $ F_i \supseteq L$,
		there is a subset $ R^\prime \sbe R $ with the same Zariski closure $ \overline{R^\prime} = \overline{R} $ 
		such that $ h_{F_i}^{aff} $ is bounded on $ R^\prime \sbe X_{F_j}(F_j) $.
	\end{enumerate} 
	Then  $ R $ is finite. 
\end{theorem}

Consequently, we prove the following result on the persistence conjecture. 

\begin{corollary}[Main Result IV, Evidence for the persistence conjecture] \label{cor:EvertseGyoryAffine}
	Let $X\sbe \Aa^n_{\oQ}$ be an arithmetically hyperbolic affine variety over $ \oQ $,
	let $K$ be a number field, let $S$ be a finite set of places of $K$, let $ \mathcal{X} \sbe \Aa^n_{ \OKs}$ be a model for $X$ over $ \OKs$. 
	Assume that for every  $\ZZ$-finitely generated integral domain $A\spe \OKs$   with quotient field $L$ and every finite set of function fields $ \{ F_1, \dots, F_t \} $ with $ F_i \supseteq L$, there is a subset $ R^\prime \sbe \mathcal{X}(A) $ with $ \overline{R^\prime} = \overline{\mathcal{X}(A)} $ 
		such that $ h_{F_i}^{aff} $ is bounded on $ R^\prime \sbe X_{F_j}(F_j) $.
	Then  $ X $ is absolutely arithmetically hyperbolic. 
\end{corollary}

\subsection*{Acknowledgements}
 
I would like to thank Ariyan Javanpeykar. He introduced me to Lang-Vojta's conjecture. I am very grateful for many inspiring discussion and his help in writing this article. I gratefully acknowledge the support of
SFB/Transregio 45.

\section{Height} \label{s:height}
 
Following \cite[Part~B]{Sil00}, in this section, we introduce the notion of height on projective varieties over number fields and function fields. 

\subsection{Height on projective space over a number field }

Let $K$ be a number field. The \emph{set of places} $M_K $ is the union of the \emph{set of finite places} $M_K^0$ and the \emph{set of infinite places} $M_K^\infty$. The set $M_K^0$ is the set of all non-zero prime ideals  $ \idp \sbe \OK $, and the set $M_K^\infty $ is the set of all real places, i.e., embeddings $ \sigma\colon K \to \RR $ and complex places, i.e., pairs of complex conjugated embeddings $ \tau,\bar{\tau}\colon K \to \CC$. Each place $ p \in M_K $ comes with an associated \emph{absolute value} $ \vert\vert . \vert\vert_p \colon K \to \RR $ given by 
\begin{eqnarray*}
	\vert\vert x \vert\vert_p = &\left(N_{K/\QQ} (\idp) \right)^{-\ord_{\idp}(x)}   & \text{ if }   p = \idp \in M_K^0, \\
	\vert\vert x \vert\vert_p = &\vert \sigma (x) \vert  & \text{ if }  p = \sigma  \text{ is real, and} \\
 	\vert\vert x \vert\vert_p = &\vert \tau(x) \vert^2  & \text{ if }  p = (\tau,\bar{\tau}) \text{ is complex}  
\end{eqnarray*}
for $ x \in K $. 
For  a point $ Q = (x_0:\ldots:x_n) \in \PP^n(K) $, we define the \emph{height} 
$$
 	H_K(Q) =   \prod_{p \in M_K} \max_i
 	 \{ \vert\vert x_i \vert\vert_p  \} . 
$$
We note that the height function over number fields has the following useful property; for proof see \cite[Theorem~B.2.3]{Sil00}.

\begin{theorem}[Northcott Property] \label{t:northcott property}
	Let $ K $ be a number field. If $H_K$ is bounded on a subset $ R \sbe \PP^n(K) $, then $ R $ is finite. 
\end{theorem}

\subsection{Height on projective space over a function field}\label{ss:height over functions fields}

Let $k$ be an algebraically closed field of characteristic zero. An \emph{algebraic function field} over $k$ is a field extension $F \spe k $ of transcendence degree $1$. Equivalently, $F = K(C)$ is the function field of a smooth projective curve $C$  over $k$. The \emph{set of places on $F$} is the set of all closed points $ M_F = C(k) $. Note that every algebraic function field over $k$ is a finite field extension of $ k(t) $ and vice versa. For $ p \in M_F $, the \emph{absolute value associated to $p$} is the function   
$ \vert\vert . \vert\vert_p \colon F \to \RR $ given by 
$$
	\vert\vert x \vert\vert_p = e^{-\ord_p(x)}
$$
for $x\in F$. For 
$ Q = (f_0:\ldots:f_n) \in \PP^n(F)$, we define the \emph{height} 
$$
H_F(Q) =   \prod_{p \in M_F} \max_i
\{ \vert\vert f_i \vert\vert_p  \} .
$$
\subsection{Weil's height machine} 
Let $L = \QQ $ or $ L = k(t) $. 
For any finite field extension $ K \spe L $, we defined the height function $ H_K$ above. 
The \emph{absolute height} is the function $ H\colon \PP^n\left(\overline{L}\right) \to \RR $, that maps a point $ Q = (x_0:\ldots:x_n) \in \PP^n\left(\overline{L}\right) $ to $ H_K(Q)^{1/[K:L]}$, where $K\spe L$ is a finite field extension such that $ x_i \in K$ for all $i$. 

Consider the embedding $\iota\colon \Aa^{n} \to \PP^n\colon (x_1,\dots,x_n) \to (1:x_1: \ldots: x_n) $.
For each height function, there is an associated \emph{affine height function}
$H_K^{aff} := H_K \circ \iota $, $ H^{aff} := H \circ \iota $ defined on points in affine space.
Furthermore, for each height function, there is an associated \emph{logarithmic height function}
$h_K := \log \circ H_K $, $h := \log \circ H $, $h_K^{aff} := \log \circ H_K^{aff} $, $h^{aff} := \log \circ H^{aff} $.

As explained in \cite[Theorem~B.3.2, Remark~B.3.2.1, Theorem~B.10.4]{Sil00}, there is the following way of defining height functions on projective varieties.

\begin{theorem}[Weil's height machine]\label{t:weils height machine}
 For each projective variety $X$ over $\overline{L} $ and each Cartier divisor $D$ on $V$, there is an associated function
$$
	h_{X,D}\colon X\left( \overline{L} \right) \to  \RR
$$ 
such that the following properties are satisfied ($O(1) $ denotes some bounded function on $ X\left(\overline{L} \right)$).
\begin{enumerate}
	\item Let $ H \sbe \PP^n $ be a hyperplane, and let $ h $ be the absolute logarithmic height on projective space. Then $ h_{\PP^n,H}  = h  + O(1)$. 
	\item  Let $ \phi \colon X \to Y$ be a morphism of projective varieties over $ \overline{L} $, and let $ D \in \CaDiv(Y) $. Then
	$ h_{X,\phi^*D}  = h_{Y,D} \circ \phi   + O(1).$
	\item Let $ D,E \in \CaDiv(X) $. Then
	$ h_{X,D+E}  = h_{X,D}  + h_{X,E} + O(1). $
	\item  Let $ D,E \in \CaDiv(X) $ with $ D \sim E $ linearly equivalent. Then
	$ h_{X,D}  = h_{X,E}  + O(1).   $
	\item If $ D  \in \CaDiv(X) $ is base point free, then
	$ h_{X,D} = h \circ \phi_D + O(1). $
\end{enumerate}
\end{theorem}

\section{Corvaja-Zannier for countable function fields}

In this section, we will prove Theorem \ref{t:corvaja zannier function fields}.
Let \ki be a countable algebraically closed field of characteristic 0. Note that $\oQ \sbe k$. Let $ F $ be an algebraic function field over $k$ as defined in Subsection \ref{ss:height over functions fields}. 
Let $X$ be a projective variety over $\oQ$ and let $D$ be a very large divisor on $X$.
We choose a basis  $ \phi_0, \dots, \phi_n $ of $ V(D) := \mathrm{H}^0(X,\mathcal{O}_X(D)) $, where $n = \dim_{\oQ} \mathrm{H}^0(X,\Oc_D) - 1 $. Let $ \phi_D\colon X \to \PP_{\oQ}^n $ denote the associated rational map. 

\subsection{Construction of the exceptional locus} 
\label{ss:exceptional locus}
We will start by constructing a proper closed subscheme $ Z \sbe X $ depending only on $D$ that "contains almost all $D$-integral points". 
Since $D$ is very large, there is a finite index set $J$ and, for each $j\in J $, there is a basis $ B_j = \{\psi_{j,0}, \dots, \psi_{j,n} \} $ of $ V(D) $ satisfying the following. Given any point $P \in D\left(\overline{F}\right)$, there is an index $ j \in J $ such that, for all irreducible components $E$ of $D$ that contain $P$, we have   
$$ \ord_E \left( \prod_{\psi \in B_j} \psi \right) > 0. $$
We can take for example $ J = D\left( \overline{F} \right)/ \sim $, where $ P \sim Q $ if $P$ and $Q$ are contained in all the same irreducible components of $D$. 

For each $j \in J$ and $i\in \{0,\dots,n\}$, there is a linear form $ \Lambda_{j,i} \in \oQ[x_0,\dots,x_n]_1 $ such that $ \psi_{j,i} = \Lambda_{j,i}(\phi_0,\dots,\phi_n) $. Let $ H_{j,i} \sbe \PP^n_{\oQ} $ denote the hyperplane determined by $\Lambda_{j,i} = 0$ and let $H_i \sbe \PP^n_{\oQ}$ denote the hyperplane determined by $ x_i = 0 $. Let $ \mathcal{H} = \{ H_{j,i}\}_{j,i} \cup \{H_i\}_i $.
   
We are going to apply Wang's version of Schmidt's subspace theorem (see \cite[p.~821]{Wan04} or \cite{AnWang}) on the collection $\mathcal{H}$. For its statement, we need the following terminology. 
Let $H \sbe \PP^n_F $ be a hyperplane given by a linear form $ \Lambda $ and let $v_p = \ord_p $ denote valuation associated to some place $p \in M_F$. Then the Weil function associated to $(H,p)$ is the map 
$$
\lambda_{H,p}\colon \mathbb{P}^n(F)  \to \RR
$$
where for a point $ P \in \mathbb{P}^n(F) $  with coordinates $ (x_0,\dots,x_n) \in F^{n+1} \setminus \{ 0\} $, we set 
$$
\lambda_{H,p}(P) 
=	 v_p(\Lambda(x_0,\dots,x_n)) - \min\limits_{0\leq i \leq n} \{ v_p(x_i) \}. 
$$
\begin{theorem}[Wang's subspace theorem for function fields] \label{thm:sstWang}
	Let $F$ be a function field and let $ \mathcal{H} $ be a finite set of hyperplanes in $ \mathbb{P}_F^n $.  
	Then there is a finite union of proper linear subspaces $Y\sbe \PP_F^n$
	that may be constructed from elements of $ \mathcal{H} $ using only operations $ <.,.> $ and $ \cap $, 
	satisfying the following.
	
	For all finite sets of places $ T \sbe M_F$ and all $ \epsilon > 0  $, 
	there are constants $ C,C' \in \RR$ such that 
	for any $ P \in \mathbb{P}^n(F)  $, at least one of the following statements holds.
	\begin{enumerate}
		\item The point $P$ lies in $Y$.
		\item The height of $P$ satisfies the inequality $$ h_F(P) \leq C.$$
		\item The height of $P$ satisfies the inequality
		$$
		\sum\limits_{p \in T}	\max_I \sum_{H \in I} \lambda_{H,p}(P) \leq (n+1+\epsilon)h_F(P) + C^\prime,
		$$
		where the maximum is taken over all subsets
		$ I = \{H_1, \dots, H_m \} \sbe \mathcal{H} $ such that the linear forms defining the $ H_i $ are linearly independent.
	\end{enumerate}
\end{theorem} 

Let $ Y \sbe \PP^n $ be the union of linear subspaces yielded by Theorem \ref{thm:sstWang} applied to the collection $\mathcal{H}$
above. Let $ Z \sbe X $ be the union of the Zariski closure of $ \phi_D\inv(Y) $ and the locus of indeterminacy of $ \phi_D $.
Note that considering the way it was constructed, we see $Z$ is again a scheme over $\oQ$.

\subsection{Binding or bounding $D$-integral points}\label{ss:binding or bounding integral points}

For a finite subset of places $ T \subset M_F $, we define
$$
\OFt= \{ x \in F \,\vert\, v_p(x) \geq 0 \text{ for all } p\in M_F\setminus T \}.
$$
Note that the subring $ \OFt \subseteq F $ is integrally closed. 
A subset $R\sbe (X \setminus D)(F) $ is called a set of \emph{
	\dtint points}
if there is a model $\Wm$ for $ W = X \setminus D $ over $ \OFt $ such that $ R \sbe  \Wm(\OFt) \sbe (X\setminus D)(F) $.
We claim the following. 

\begin{lemma}\label{l:1.5}
	Let $R\sbe (X \setminus D)(F) $ be a set of \dtint points. Then $ R $ contains a subset $ R^\prime \subseteq R $ with the same Zariski closure $ \overline{R} = \overline{R^\prime} $ such that $ h \circ \phi_D $ is bounded on $ R^\prime \setminus Z $, where $Z$ is the exceptional locus constructed in Subsection \ref{ss:exceptional locus}.
\end{lemma}

In the remainder of this section, we will prove Lemma \ref{l:1.5}.
The set $R$ is a finite union of subsets that have an irreducible Zariski closure. 
We can consider each one of these subsets individually and therefore assume that the Zariski closure $ \overline{R}$ is irreducible. For each place  $p \in M_F$, there is an absolute value $ \vert\vert . \vert\vert_p = exp \circ(-v_p(.)) $ on $F$.
We fix some embedding $ X_F \sbe \PP_F^m $. The absolute value $\vert\vert . \vert\vert_p$ defines a topology on projective space $\PP_F^m(F)$ and so on $X(F)$. 
By $ F_p $ we denote the completion of $F$ with respect to $ \vert\vert . \vert\vert_p $.
\begin{lemma} \label{l:converging dense sequence} If the Zariski closure $ \overline{R}$ is irreducible, then there is a sequence $ (Q_i)_{i \in \NN } $ of points in $R$ and, for each $ p \in T $, there is a point $ Q_p \in X(F_p) $ such that
	\begin{enumerate}
		\item $ \{ Q_i \vert i \in \NN \} $ is Zariski dense in $ \overline{R} $, and
		\item for all $p\in T $, the sequence  $Q_i$ converges towards $ Q_p $ in the  $\vert\vert . \vert\vert_p$-topology.
	\end{enumerate}
\end{lemma}

\begin{proof} 
	We fix some $ p \in T $. Let $ V = \overline{R} $ be the Zariski closure and let $ M \sbe X(F_p) $ be the closure of $R$ in the $\vert\vert . \vert\vert_p$-topology. The set $M$ is compact. 
	Hence we can find a point $ Q_p \in X(F_p)   $ such that for each  $\vert\vert . \vert\vert_p$-open  neighbourhood U of $ Q_p $, the set $ U \cap R $ is Zariski dense in $V\left(\overline{F}\right)$. 
	This is because otherwise $R$ would be the union of finitely many sets that are not Zariski dense in $ V\left(\overline{F}\right) $,
	which is not possible because $V$ is irreducible.
	
	Since $F$ is countable, there are countably many hypersurfaces of $ \PP^n_F $ not containing $V$. 
	Let $ (H_i)_{i \in \NN} $ be an enumeration of these. 
	Let $ U_1 \supseteq U_2 \supseteq \dots $ be a descending chain of $\vert\vert . \vert\vert_p$-open neighbourhoods of $Q_p$	with $ \bigcap U_i = \{ Q_p \} $. 
	Since $R$ is Zariski dense in $V\left(\overline{F}\right)$, the set $ R \cap U_i $ is not contained in $H_i\left(\overline{F}\right)$. 
	Choose $ Q_i \in R \cap U_i \setminus H_i $. 
	Then the sequence $ Q_i $ converges towards $ Q_p $ in $\vert\vert . \vert\vert_p$-topology and is Zariski dense in $V$. 
	
	Iterating this process for all other $ p \in T $ finishes the proof.
\end{proof}

We set $ R^\prime $ to be the sequence constructed in Lemma \ref{l:converging dense sequence}. Let $Z$ be as constructed in Subsection \ref{ss:exceptional locus}. Let $R'' = R' \setminus Z $. Then the rational map $ \phi_D $ is determined on $ R'' $. We have to show that the height $h$ is bounded on $ \phi_D(R'') $. Note if $ R'' $ is finite, we are done. Otherwise, $R''$ is a subsequence of $R'$ and has the same convergency properties as $R'$ described in Lemma \ref{l:converging dense sequence}.  

By Theorem \ref{thm:sstWang}, given $\epsilon>0$, we can find constants $C,C'$ such that for all $ Q \in \phi_D(R'') $ with $h_F(Q) > C$,
the inequality 
$$
\sum\limits_{p \in T}	\max_I \sum_{H \in I} \lambda_{H,p}(Q) \leq (n+1+\epsilon/2)h_F(Q) + C^\prime,
$$
is satisfied. Let $ C'' \in \RR $ be arbitrary. Then for $ P \in R'' $ with $ h_F(\phi_D(Q)) > \frac{2(C'-C'')}{\epsilon} $,
we have 
$$ (n+1+\epsilon/2)h_F(Q) + C^\prime < (n+1+\epsilon)h_F( Q) + C''. $$
That means by changing $ C $ if necessary we can choose $C'$ arbitrarily. 
Now since $ R'' $ is a set of \dtint points, there is an $ a \in \OFt $ 
so that for all $ Q \in R'' $ and $ i \in \{0,\dots,n \} $, we have $ a \phi_i(Q) \in \OFt $. 
This implies
\begin{align*}
	h_F( \phi_D(Q) ) = h_f(a\phi_D(Q)) 
	&= 		- 	\sum\limits_{p\in M_F} \min\limits\{ v_p(a\phi_0(Q)), \dots, v_p(a\phi_{n}(Q))  \}		
	\\		
	&\leq 	- 	\sum\limits_{p \in T} \min\limits\{ v_p(a\phi_0(Q)), \dots, v_p(a\phi_{n}(Q))  \}
	\\
	&\leq	- 	\sum\limits_{p \in T} \min\limits\{ v_p(\phi_0(Q)), \dots, v_p(\phi_{n}(Q))  \}
	-	\sum\limits_{p \in T} v_p(a)
\end{align*}
The term $ \sum\limits_{p \in T} v_p(a) $ is some constant. All together, we have seen that given $\epsilon>0$ and $ C' \in \RR $, we can find a constant $ C\in \RR $ such that for all $ Q \in R'' $, we either have 
\begin{equation} \label{e:bound}
		\sum\limits_{p \in T}	f_p(Q)  \leq  C^\prime,
\end{equation}
where 
$$
	f_p(Q) = \max_I \left( \sum_{H \in I} \lambda_{H,p}(\phi_D(Q)) \right) + (n+1+\epsilon) \cdot \min\limits\{ v_p(\phi_0(Q)), \dots, v_p(\phi_{n}(Q))  \}
,$$
or we have $ h_F(\phi_D(Q)) \leq C$.
In other words, this means if we find a lower bound for the left-hand-side of Inequality \ref{e:bound}, i.e., a constant $C'$ such that Inequality \ref{e:bound} is never satisfied, then we have bounded the height.

We put $$ M = \max \left\{ - ord_E( \phi_j )\,\middle\vert\, \text{E is an irreducible component of D, j } \right\} $$ and choose $ \epsilon = 1/M $. We can bound each function $f_p$ separately. So fix $ p \in T $, let $ Q_p $ be as in Lemma \ref{l:converging dense sequence}.  

First assume  $ Q_p \not\in D(F_p) $.  Considering $ I = \{H_0,\dots, H_n \}$, we see \begin{align*}
	f_p(q) 
	&\geq \left( \sum_{i= 0}^n \lambda_{H_i,p}(\phi_D(Q)) \right) + (n+1+\epsilon) \cdot \min\limits \{ v_p(\phi_0(Q)), \dots, v_p(\phi_{n}(Q))  \} \\
	&=\left( \sum_{i= 0}^n  v_p(\phi_i(Q)) \right) + \epsilon \min\limits_{0\leq i \leq n} \{ v_p(\phi_i(Q)) \}.
\end{align*}
As all $ \phi_i $ have no pole at $ Q_p $, the right-hand-side converges to the value at $ Q_p $. In particular, it is bounded.

Now assume $ Q_p \in D(F_p) $. We choose $j \in J$ such that for each irreducible component $E$ of $D$ with $ Q_p \in E(F_p) $, we have 
$$ \ord_E \left( \prod_{\psi \in B_j} \psi \right) > 0. $$
Considering $ I = \{H_{j,0},\dots, H_{j,n}\} $, we see
\begin{align*}
	f_p(Q) 
	&\geq \left( \sum_{i= 0}^n \lambda_{H_{(j,i)},p}(\phi_D(Q)) \right) + (n+1+\epsilon) \cdot \min\limits \{ v_p(\phi_0(Q)), \dots, v_p(\phi_{n}(P))  \} \\
	&=\left( \sum_{i= 0}^n  v_p(\psi_{j,i}(Q)) \right) + \epsilon \min\limits_{0\leq i' \leq n} \{ v_p(\phi_{i'}(Q)) \}.
\end{align*}
For each irreducible component $E$ of $D$ with $ Q_p \in E(F_p)$, we have 
$$
ord_E \left( \phi_{i'} \cdot \left( \prod\limits_{i=0}^{n}  \psi_{j,i}\right)^M \right)
= 	ord_E(\phi_{i'}) + M ord_E \left(\prod\limits_{i=1}^{l(D)}  \psi_{j,i} \right)
\geq	- M + M = 0.
$$
This shows that $ f_p $ is bounded by the minimum of a finite set of functions that have no pole along $D$. As before by convergency reasons, this shows $f_p$ is bounded from below on $ R''$. This finishes the proof of Lemma \ref{l:1.5} and Theorem \ref{t:corvaja zannier function fields}.

\section{Evertse-Gy\H ory's method}

\label{Chapter4}

In this section, we review Evertse-Gy\H ory's method \cite[Chapter~8]{Eve15} and prove Theorem \ref{thm:EvertseGyoryAffine}. 

\subsection{Degree and height functions on finitely generated domains}\label{s:degree and height functions}

Let $A$ be a $\ZZ$-finitely generated domain. We choose a maximal number of algebraically independent elements $z_1,\dots,z_q \in A$. Then $A$ is an extension of $A_0 = \ZZ[z_1,\dots,z_q]$, and we can find elements $ \yt \in A $ that are algebraic over $ A_0 $ such that $ A = A_0[ \yt ] $. 
We denote the quotient field of $ A $ by $ L = Q(A) $ and the quotient field of $ A_0 $ by $ L_0 = Q(A_0) = \QQ(\zq) $.
Note that the field extension $ L/L_0 $ is finite. 
By the primitive element theorem, there is a $ y \in L $ such that $ L = L_0(y) $.
The primitive element $y$ has a minimal polynomial
$$ G(\zq)(x) = x^d + G_1(\zq) x^{d-1} + \dots + G_d(\zq) \in L_0[x]. $$ 
Let $ g \in A_0 $ be the common denominator of the rational functions $G_i \in L_0 $. By replacing $ y $ with $ y \cdot g$, we may assume that $ G \in A_0[x] $.
The elements $ 1, y, \dots, y^{d-1} $ form a basis of $ L $ as $ L_0 $ vector space. Therefore, for each $ \alpha \in L $, we can find 
$ Q_\alpha,\Pda \in A_0 $ (unique up to sign) with no common factor such that 
\begin{equation}\label{eq:presentation}
	\alpha = {Q_\alpha}^{-1} \sum\limits_{j=0}^{d-1}  P_{j,\alpha} y^j. 
\end{equation}
In particular, we can find such for $ \alpha \in \{ \yt \} $. We set $ f = \prod_{i=1}^t Q_{y_i} \in A_0$. 
Then there is an inclusion of rings
\begin{equation} \label{eq:niceA}
	A \sbe A_0[f^{-1}, y ] = : B.
\end{equation}
The ring $B$ has the same quotient field as $A$. For $ \alpha \in L $, we define
\begin{align*}
	\overline{\deg}(\alpha) &:=  \max \{ \deg(P_{\alpha,0}), \dots , \deg(P_{\alpha,d-1}), \deg(Q_\alpha) \}
	\\ 	\overline{h}(\alpha) &:=  \max \{ h^{aff}(P_{\alpha,0}), \dots , h^{aff}(P_{\alpha,d-1}), h^{aff}(Q_\alpha) \}
\end{align*}
where, for $ g = \sum_\mu c_\mu z_1^{\mu_1}\dots z_q^{\mu_q} \in A_0 $, we set 
$ \deg\left( g  \right)  = \max\{ \mu_1 + \dots + \mu_q | c_\mu \neq 0  \} $  
and
$h^{aff}(g) =   \sum_{p \in M_\QQ } \ln \max\limits_{\mu, c_\mu \neq 0} \{ 1, ||c_\mu||_p \} $.
Furthermore, for $ (\alpha_1, \dots \alpha_n) \in L^n $, we define: 
\begin{align*}
	\overline{\deg}(\alpha_1, \dots \alpha_n) &:=  \max \{ \odeg(\alpha_1), \dots , \odeg(\alpha_n) \}
	\\ 	\overline{h}(\alpha_1, \dots \alpha_n) &:=  \max \{ \oh(\alpha_1),\dots \oh(\alpha_n) \}
\end{align*}
These functions combined have the Northcott property.
\begin{lemma}\label{l:deg+h has Northcott property}
	Let $ R \sbe L^n $ be a subset such that $ \odeg $ and $ \oh $ are bounded on $R$. Then $ R $ is finite.
\end{lemma}
\begin{proof}
	Let $ \alpha  = {Q_\alpha}^{-1} \sum_{j=0}^{d-1}  P_{j,\alpha} y^j \in L $ be a coordinate of some point in $R$. By Theorem \ref{t:northcott property},
	the boundedness of $ \oh $ on $R$ implies that there are only finitely many possibilities for the coefficients of $ \Qa, \Pao , \dots , \Pad  $.
	Since $ \odeg $ is bounded on $ R$, also the degree of these polynomials is bounded. 
\end{proof}

\subsection{Using embeddings into function fields to bound the degree} \label{ss:embeddings into function fields}

We may embed $L$ into function fields. We denote the $d$ many $L_0$-invariant embeddings of $ L $ into an algebraic closure of $ L_0$ by $ x \mapsto x^{(j)} $, where $ j \in \{1,\dots, d\} $.
For $ i \in \{1,\dots,q\} $, let $ k_i $ be an algebraic closure of $ \QQ(z_1,\dots,\hat{z_i},\dots,z_q) $. The notation $ \hat{z_i} $ means that we leave out $ z_i $. 
Then $ F_i = k_i(z_i,y^{(1)},\dots,y^{(d)}) $ is an algebraic function field of transcendence degree 1 over $ k_i $. 
Consider the $d$ many different embeddings of $L$ into $F_i$
\begin{equation} \label{eq:FFieldEmbeddings}
	\varphi_{i,j} \colon L \to F_i\colon z_i \mapsto z_i, \, y \mapsto y^{(j)}, \quad j \in \{1,\dots, d\}.
\end{equation}
Evertse and Gy\H ory showed the following; see \cite[Lemma~8.4.1]{Eve15}.

\begin{lemma} 
	\label{lem:degSmalerThanHeight}
	There is a constant $C > 0$ such that for all $ \alpha \in L $
	$$
	\overline{deg}(\alpha) \leq C + \sum\limits_{i=1}^q [F_i:k_i(z_i)]^{-1} \sum\limits_{j=1}^d h_{F_i} (\alpha^{(j)}).
	$$
\end{lemma}

\begin{corollary}\label{cor:functionField}
	Let $ R \sbe L^n $ be a subset such that  $ h^{aff}_{F_i} $  is bounded from above on 
	$$ \varphi_{i,j}(R) := \{ (\varphi_{i,j}(\alpha_1), \dots, \varphi_{i,j}(\alpha_n)) | (\alpha_1,\dots,\alpha_n) \in R\} \sbe F_i^n $$
	for all $ i \in \{1,\dots,q \} $ and $ j \in \{1, \dots d \} $.
	Then $ \odeg $ is bounded on $R$.
\end{corollary}

\begin{proof}
	Let $(x_1,\dots, x_n) \in \varphi_{i,j}(R)  $. For all $ j \in \{1,\dots,n\} $, we have
	\begin{align*}
		h^{aff}_{F_i} (x_1,\dots, x_n)  
		=	\sum_{p \in M_{F_i}} \max\{0,-v_p(x_1),\dots,-v_p(x_n) \} 
		\geq\sum_{p \in M_{F_i}} \max\{0,-v_p(x_j)\} 
		=  	h_{F_i}^{aff}(x_j). 
	\end{align*}		
	So the height is bounded on each coordinate. By Lemma \ref{lem:degSmalerThanHeight}, this implies that $ \odeg $ is bounded on each coordinate.
\end{proof}

\subsection{Specializations}\label{ss:specialization}

As we have seen in the Subsection \ref{s:degree and height functions}, the finitely generated domain $ A $ is contained 
in $ B = A_0[f^{-1},y] $, where $ f \in A_0 = \ZZ[\zq] $ and $ y $ is integral over $ A_0 $
with minimal polynomial
$$ G(\zq)(x) = x^d + G_1(\zq) x^{d-1} + \dots + G_d(\zq) \in A_0[x]. $$  

We want to define specialization maps $ B \to \overline{\QQ} $. 
The discriminant $ \disc(G) $ is a polynomial in the variables $ \zq $. Since $ G $ is irreducible and separable, this polynomial does not vanish entirely. We define
\begin{equation}\label{eq:H}
	H = G_d \cdot \text{disc}(G) \cdot f \in A_0.
\end{equation} 
Now for $ u \in \ZZ^q $ with $ H(u) \neq 0 $,
the polynomial $ G(u)(x) \in \ZZ[x] $ has $ d $ many distinct roots $ y_{u,1}, \dots, y_{u,d} \in \overline{\QQ} \setminus \{0\}  $ that are integral over $ \ZZ $, and  we have $ f(u) \in \ZZ \setminus \{ 0 \} $.
Now mapping $ z_j \mapsto u_j $ and $ y \mapsto y_{u,i} $ defines a map
\begin{equation}\label{eq:specialization}
	\psi_{u,i} \colon B \to \ZZ[f(u)^{-1},y_{u_i}] \sbe \mathit{O}_{K_{u,i},S_{u,i}} \sbe K_{u,i},
\end{equation}
where $ K_{i,u} = \QQ(y_{u,i}) $ is a number field and $ S_{u,i} $ is the set that contains all infinite places and the $  \idp \in M_{K_{u,i}}^0 $ with $ \ord_\idp(f(u)) > 0 $ or $ \ord_\idp(y_i(u)) < 0$.

\begin{rem}\label{rem:specialization induces iso}
	Let $K$ be a number field and let $S\sbe M_K$ be a finite such that there is an embedding $ \OKs \sbe B $.
	Then $$ \idp = \ker\left( \mathit{O}_{K,S} \xrightarrow{\psi_{u,i}} K_{u,i} \right)  \sbe \mathit{O}_{K,S}$$
	is a prime ideal. Hence either $ \idp = 0 $ or $ \idp $ contains some prime number $ p \in  \ZZ $.
	The latter is impossible, as $K_{u,i} $ is no $ (\ZZ/p\ZZ)$-algebra. Hence the specialization map $\psi_{u,i}$ induces an isomorphism of $K$ with some subfield of $ K_{u,i} $.
\end{rem}

We will use the following version of a result by Evertse and Gy\H ory.

\begin{proposition}\label{p:height + height bound = finite dimension 1}
	Let $ R \sbe B $ be a subset of bounded $ \odeg $,
	let $ N \in \NN $ be so big that $ \odeg(r) \leq N $ for all $ r \in R $ and $ deg(H) \leq N  $, and let 
	$$
	\mathcal{U}= \{u \in \ZZ^q | \max_i\{ |u_i| \} \leq N , H(u) \neq 0 \}.
	$$
	Furthermore, suppose the height function $ H_{K_{u,i}}^{aff} $ is bounded on $ \psi_{u,i}( R ) $ for all $ i \in \{1,\dots,d\}  $ and $ u \in \mathcal{U} $.
	Then $ R $ is finite.
\end{proposition}

\begin{proof}
	By \cite[Lemma~8.5.6]{Eve15}, $\overline{h}$ is bounded on $R$. The assertion now follows from 
	Lemma \ref{l:deg+h has Northcott property}.
\end{proof}

\begin{corollary}\label{c:height + height bound = finite dimension q}
	Let $ R \sbe B^n $ be a subset of bounded $ \odeg $, 
	let $ N \in \NN $ be so big that $ \odeg(r) \leq N $ for all $ r \in R $ and $ deg(H) \leq N  $, and let 
	$$
	\mathcal{U}= \{u \in \ZZ^q | \max_i\{ |u_i| \} \leq N , H(u) \neq 0 \}.
	$$
	Furthermore, suppose for all $ i \in \{1,\dots,d\} $ and $ u \in \mathcal{U} $, the height function $ H_{K_{u,i}}^{aff} $ is bounded on 
	$$ \{ (\psi_{u,i}(x_1),\dots,\psi_{u,i}(x_n))  | (\xn)\in R \}. $$
	Then $ R $ is finite.
\end{corollary}

\begin{proof}
	When $ \odeg $ is bounded on $ R $ by $ N $, then $ \odeg $ is also bounded by $ N $ on each coordinate of every point in $ R $.
	Furthermore, for $ (q_1,\dots, q_n) \in K_{u,i}^n $, we have
	\begin{align*}	
		H_{K_{u,i}}^{aff}(q_1,\dots, q_n) = \prod\limits_{p \in M_{K_i}} \max \{ 1, ||q_1||_p, \dots, ||q_n||_p \} 
		\geq  \prod\limits_{p \in M_{K_i}} \max \{ 1, ||q_j||_p \} =  H_{K_{u,i}}^{aff}(q_j).
	\end{align*}
	Hence, the set of coordinates of points in $ R $ satisfies the assumptions of Proposition \ref{p:height + height bound = finite dimension 1}. So we are done.
\end{proof}

\subsection{Formulation of the method}

Let $ R \sbe A^n $ be a subset, where $A$ is a $\ZZ$-finitely domain such that
\begin{enumerate}
	\item for all embeddings into function fields 
	$ \varphi_{i,j}\colon A \to F_i $ (see Subsection \ref{ss:embeddings into function fields}), the height $ H_{F_i}^{aff} $ is bounded on $ \varphi_{i,j}(R) $, and
	\item for all specialization maps $ \psi_{u,i}\colon A \to \mathit{O}_{K_{u,i},S_{u,i}} $(see Subsection \ref{ss:specialization}), the height $ H_{K_{u,i}}^{aff} $ is bounded on $\psi_{u,i}(R)$.
\end{enumerate}

Then Evertse and Gy\H ory's results Corollary \ref{cor:functionField} and Corollary \ref{c:height + height bound = finite dimension q} combined prove the finiteness of $ R $. 

\begin{proof}[Proof of Theorem \ref{thm:EvertseGyoryAffine}]By Remark \ref{rem:specialization induces iso}, the set  $ \psi_{u,i}(R) $ is contained in some set of integral points of $ X' = X\otimes_{K,\psi_{u,i}} K_{u,i} $. The variety $X'$ is arithmetically hyperbolic as well; see \cite{JLalg}. Hence the sets  $ \psi_{u,i}(R) $ are all finite and therefore the height is bounded. 
Now we apply Evertse-Gy\H ory's method.
\end{proof}

\section{Absolute arithmetic hyperbolicity}

In this section, we will prove Theorem \ref{thm:I}.

\begin{lemma}\label{lem:levinLarge}
	Let $m$ be an integer, let $k$ be a field of characteristic zero, let $X$ be a projective variety over $k$, and let $ D = \sum_{i=1}^r D_i  $ be a sum of effective, big and nef divisors on $ X $ 
	such that at most of $m$ of the $ D_i $ meet in any point, where $ r > 2m \dim(X) $. 
	Then there is a non-singular projective variety $ X^\prime $ over $k$ and a birational morphism $ \pi\colon X^\prime \to X $
	and positive integers $ b_i $ such that
	\begin{enumerate}
		\item the divisor $ E = \sum_{i=1}^r b_i D_i $ has the same support as $ D $,
		\item the pullback $E^\prime = \pi^* E $ is very large,
		\item the associated rational map $ \phi_{E^\prime} $ is birational onto its image and
		\item the variety $ X^\prime$ and every irreducible component of $ E^\prime $ is non-singular.  [Lemma~4.14]
	\end{enumerate}     
\end{lemma}

\begin{proof}
	Let $q = \dim(X) $.
	By Hironaka, we can find a smooth projective $X'$ and birational morphism $ \pi\colon X^\prime \to X $ such that $ D_i' :=  \pi^* D_i $  is non-singular for every $i$.  
	The $ D_i^\prime $ are big and nef again, since they are pullbacks of big and nef divisors along a birational morphism. 
	Therefore, all $q$-fold intersections of the $ D_i^\prime $ are nonnegative and $ {D_i^\prime}^q > 0 $.
	By \cite[Lemma~9.7]{Levin}, the divisor $ \sum_{i=1}^r D_i' $ is equidegreelizable.
	Hence by \cite[Theorem~9.9]{Levin}, we can find suitable $b_i$.
\end{proof}

\begin{proposition}\label{lem:Nicelem}
	Let $m$ be an integer, let $ X $ be a projective variety over $ \oQ $,
	and let $ D = \sum_{i=1}^r D_i$ be a Cartier divisor on $X$, where all the $ D_i $ are effective and ample and at most $m$ of the $ D_i $ meet at any point, with $ r > 2m\dim(X)$.
	Then there is a proper closed subvariety $ Z \sbe X $ and an ample Cartier divisor $ E $ with the same support as $ D $ such that the following is true.
	
	For any countable algebraic function field $ F $, for any finite set of places $ T \sbe M_F $, and for all sets $ R $  of $(E,\OFt)$-integral points
	on $X_F$, 
	there is a subset $R^\prime \sbe R \setminus Z$
	such that $R^\prime $ is Zariski dense in $ R \setminus Z $ and the height function $ h_{X,E} $ 
	is bounded on $ R^\prime $.         
\end{proposition}

\begin{proof} 
	By Lemma \ref{lem:levinLarge}, we can find a Cartier divisor $ E = \sum_{i=1}^r b_i D_i $ on $ X $ 
	with the same support as $ D $ and a birational proper morphism of projective varieties $\pi\colon X^\prime \to X $
	such that $ E^\prime = \pi^* E $ is very large and  $\phi_{E^\prime}$ is birational onto its image.
	Note that $ E^\prime $ is semi-ample since it is the pullback of an ample divisor.
	Therefore, by replacing the $b_i$ with $ n b_i $ for $ n \gg 0 $, we may assume that $ E' $ is base point free. 

	There is a proper closed $ V \sbe X $ (namely the singular locus of $X$ and $D$) 
	that can be defined over $ K $ such that $ \pi\inv $ is defined outside of $ V $.    
	Let $ Z' \sbe X' $ be a proper closed subscheme like in Theorem \ref{t:corvaja zannier function fields}, and
	let $Z = \pi(Z') \cup V \sbe X $.
	
	Let $ R \sbe X(F) $ be a set of $(E,\OFt)$ integral points. Then $ \pi\inv(R \setminus V ) $ is a set of
	$(E^\prime,\OFt)$-integral points on $ X^\prime $.
	Hence, there is a subset $ R' \sbe R \setminus V  $
 	such that 
 	\begin{enumerate}
 		\item the Zariski closures of $ \pi\inv(R^\prime ) \sbe X'$ and  $  \pi\inv(R \setminus V ) \sbe X'$ equal, and
 		\item the height $ h_{X^\prime, E^\prime }$ is bounded on $ \pi\inv( R^\prime ) \setminus Z^\prime $.
 	\end{enumerate}
 	Since $\pi$ is birational and proper, the set $ Z  \sbe X $ is a  proper closed subvariety. 
 	Since $ h_{X^\prime, E^\prime } = h_{X^\prime, \pi^* E } = h_{X,E} \circ \pi + O(1) $ (see Theorem \ref{t:weils height machine}), we conclude 
 	that  $ h_{X,E} $ is bounded on $ R^{\prime}\setminus Z $. Since $ \pi $ is birational, $ R^{\prime}\setminus Z $ is Zariski dense in $ R \setminus Z$.      
\end{proof}

\begin{lemma}\label{l:induction step}
	Let $ X $ be a projective variety over $ \oQ $, and let $m$ and $r$ be positive integers with $ r>2m\dim(X)$. Let $ D = \sum_{i=1}^r D_i$ be a Cartier divisor on $X$,  where all the $ D_i $ are effective and ample and the intersection of $m+1$ distinct $ D_i $ is empty. 
	Then there is a proper closed subscheme $ Z \sbe X $ such that $X$ is absolutely arithmetically hyperbolic modulo $Z$. 
\end{lemma}

\begin{proof}	
	The finiteness of integral points is independent of the chosen model; see for example \cite[Lemma~4.8]{JLalg}. 
	Thus, let $  K $ be a number field, let  $S $ be a finite set of places and l $ \Wm $ be a model for $ W = X \setminus D$ over $\OKs$. Let $A$ be a $\ZZ$-finitely generated integral domain and let $ R = \Wm(A) $.
		
	Let $ \{F_1, \dots, F_t\}$ be a finite set of function fields. By repeatedly applying Proposition \ref{lem:Nicelem}, we can 
	find a proper closed subscheme $ Z \sbe X $, an ample Cartier divisor $E$ on $X$ with the  same support as $D$
	and a subset $ R^\prime \sbe R\setminus Z $ that is	Zariski dense in $ R \setminus Z $ such that $ h_{X,E} $ is bounded on 
	$ R^\prime \sbe (X_{F_j} \setminus E)(F_j) $ for all $j$. 
	
	We claim that $ R \setminus Z $ is finite. To show this, we apply Evertse-Gy\H ory's method.
	Let $t>0$ be an integer such that $tE$ is very ample. 
	By Theorem \ref{t:weils height machine}.3), the height $ h_{X,tE} $ is bounded on $ R^\prime $.
	We may choose an embedding $ \phi\colon X \to \PP^n $ such that $ tE = \phi^* H $ for a hyperplane $ H \sbe \PP^n $.
	The restriction of $ \phi $ to $ {X\setminus D} $ induces an embedding $ \alpha\colon X\setminus D \to \Aa^n \cong \PP^n \setminus H$.
	
	By Theorem  \ref{t:weils height machine}.1) and Theorem  \ref{t:weils height machine}.2), restricting to $ {X\setminus D} $, we have 
	$$
		h_{X,tE} = h_{X,\phi^*H} = h_{\PP^n,H} \circ \phi + O(1) = h \circ \phi +O(1). 
	$$ 
	Again restricting to $ {X\setminus D} $, gives us
	$
		h_{X,tE}\vert_{X\setminus D} = h^{aff} \circ \alpha + O(1).
	$
	Therefore, $ h^{aff}_{F_i} $ is bounded on $ R^\prime \sbe (X \setminus D)(F_i) $.
	Combined with the fact that $X\setminus D$ is arithmetically hyperbolic over $\oQ$ (\cite[Theorem~9.11A]{Levin}), it follows from Theorem \ref{thm:EvertseGyoryAffine} that $ R \setminus Z $ is finite.
\end{proof}

\begin{proof}[Proof of Theorem {\ref{thm:I}}]
	We proceed by induction on the dimension of $X$. Note any variety of dimension $0$ is absolutely arithmetically hyperbolic.
	Let $m$ be a positive integer, let $X$ be a smooth projective connected variety over $\Qbar$ of dimension $\dim(X)>0$, and let $D=\sum_{i=1}^rD_i$ be a sum of $r$ ample effective divisors on $X$ such that at most $m$ of the divisors $D_i$ meet in a point, with $r>2m\dim(X)$. 
	By Lemma \ref{l:induction step}, there is a proper closed subscheme $Z \sbe X$ that contains all but finitely many integral points. 
	We finish the proof by showing that all irreducible components of $ Y \sbe Z $ are absolutely arithmetically hyperbolic. If $ \dim(Y) = 0 $, there is nothing to do. Otherwise, we have $ 0 < \dim(Y)< \dim(X) $. Note that the pullback of the divisor $D$ along the closed immersion $ Y \to X $
	again satisfies the hypotheses of the theorem. Hence by induction $ Y $, is absolutely arithmetically hyperbolic. 
\end{proof}

\printbibliography[heading=bibintoc]

\end{document}